\documentclass[12pt]{amsart}
\usepackage{a4wide}
\usepackage{amssymb}
\usepackage{mathrsfs}
\usepackage{graphicx}
\usepackage{comment}
\usepackage{longtable}
\usepackage{hyperref}
\usepackage{color}
\usepackage{float}

\usepackage{adjustbox}
\usepackage{array}
\newcolumntype{R}[2]{%
    >{\adjustbox{angle=#1,lap=\width-(#2)}\bgroup}%
    l%
    <{\egroup}%
}
\newcommand*\rot[2]{\multicolumn{1}{R{#1}{#2}}} 

\numberwithin{equation}{section}
\def\th{\theta}
\newcommand{\comments}[1]{}

\newcommand{\be}{\begin{equation}}
\newcommand{\ee}{\end{equation}}
\newcommand{\ba}{\begin{align}}
\newcommand{\ea}{\end{align}}

\DeclareMathOperator{\Res}{Res}

\newtheorem{theorem}{Theorem}
\newtheorem{prop}{Proposition}[section]
\newtheorem{corollary}[prop]{Corollary}
\newtheorem{lemma}[prop]{Lemma}

\newtheorem{remark}[prop]{Remark}

{\begin{list}{}{%
\settowidth{\labelwidth}{\textsf{{\it #1.}}}%
\setlength{\labelsep}{4mm}%
\setlength{\leftmargin}{\labelwidth}%
\addtolength{\leftmargin}{\labelsep}%
}}%
{\end{list}}
\def\cc{,\dots,}

\def\eps{\varepsilon}
\def\la{\lambda}\def\th{\theta}

\def\lcm{\operatorname{lcm}}

{\begin{list}{}{%
\settowidth{\labelwidth}{\textsf{{\it #1.}}}%
\setlength{\labelsep}{2mm}%
\setlength{\leftmargin}{\labelwidth}%
\addtolength{\leftmargin}{\labelsep}%
\addtolength{\leftmargin}{4mm}%
\setlength{\itemsep}{6pt}%
\setlength{\listparindent}{0pt}%
\setlength{\topsep}{3pt}%
}}%

\title{Short Salem polynomials}
\author{James McKee}
\address{Royal Holloway, University of London, Egham Hill, Egham, Surrey TW20 0EX, UK}
\email{james.mckee@rhul.ac.uk}

\author{Chris Smyth} 
\address{University of Edinburgh, Edinburgh EH9 3FD, Scotland, UK}
\email{c.smyth@ed.ac.uk}


\begin{document}

\begin{abstract}
   We give a complete classification of all Salem polynomials of length $5$.
   For length $6$ we show that all but finitely many Salem polynomials lie in one of $12$ infinite families, and subject to Lehmer's Conjecture we give a complete list of the 126 exceptions.
   We provide a table of short polynomials for all known Salem numbers below the smallest Pisot number.
\end{abstract}

\maketitle

\section{Introduction and statement of results}

A \emph{Salem number} is a real algebraic integer $\tau>1$ such that all of the Galois conjugates of $\tau$ (excluding $\tau$ itself) have modulus at most $1$, and with at least one having modulus exactly $1$. 
For more about these numbers see the surveys \cite{GhateHironaka2001} and \cite{Smyth2015}.

For a polynomial $P(z)\in\mathbb{Z}[z]$ that has degree $d$, we define $P^*(z)=z^d P(1/z)$; its coefficients are those of $P$ written in reverse order. 
We speak of $P^*$ as being the \emph{reciprocal} of $P$ (its zeros are the reciprocals of the nonzero zeros of $P$).
We call $P$ \emph{reciprocal} (respectively, \emph{antireciprocal}) if $P^*=P$ (respectively, $P^*=-P$).
The \emph{length} of a polynomial is the sum of the absolute values of its coefficients.
We use $\Phi_d(z)$ for the $d$-th cyclotomic polynomial; this is the minimal polynomial of $\omega_d:=e^{2\pi i/d}$.
It will be convenient to use the term \emph{cyclotomic polynomial} more generally to mean any product of irreducible cyclotomic polynomials $\Phi_d$, allowing repeated factors, and allowing the empty product $1$.

The minimal polynomial of a Salem number $\tau$ is reciprocal, of even degree at least $4$, and all but two of its zeros lie on the unit circle (the exceptions being $\tau$ and $1/\tau$).
A {\it Salem polynomial} is defined to be the minimal polynomial of a Salem number, possibly (but not necessarily) multiplied by a cyclotomic polynomial.
The smallest known Salem number is Lehmer's number, $\lambda_0 = 1.1762808\dots$, a zero of $z^{12}-z^7-z^6-z^5+1$.
This length-5 polynomial is not the minimal polynomial of $\lambda_0$, but is a multiple of its minimal polynomial by the cyclotomic polynomial $z^2-z+1$.
It is the shortest possible multiple of the minimal polynomial of $\lambda_0$ by a cyclotomic polynomial, and we say that $\lambda_0$ has \emph{shortness} $5$, with $z^{12}-z^7-z^6-z^5+1$ being a \emph{short polynomial} for $\lambda_0$ (there is one other short polynomial for $\lambda_0$, namely $z^{14}-z^{11}-z^7-z^3+1$).

\subsection{Salem polynomials of shortness 5}
It is easy to see that if a monic integer polynomial is reciprocal and has length at most $4$, then it is cyclotomic.
Thus the shortness of a Salem number is at least $5$ (with $\lambda_0$ illustrating that $5$ is possible).

In this paper we give a complete list of all shortness-$5$ Salem numbers (Theorem \ref{T-shortness 5}).

\begin{theorem}\label{T-shortness 5} There are precisely $17$ Salem polynomials of length $5$, associated to $13$ distinct Salem numbers of shortness $5$. 
These polynomials and Salem numbers are given in Table \ref{T-1}.
\end{theorem}

The polynomials and numbers in Table \ref{T-1} match exactly the relevant polynomials and numbers in Table 1 of \cite{El-SerafyMcKee2025}.
The table in \cite{El-SerafyMcKee2025} lists all monic reciprocal primitive integer polynomials (up to change of sign of the variable) of length $5$ that have a zero of absolute value at least Lehmer's number, $\lambda_0$. 
This includes all Salem polynomials (perhaps with $z$ replaced by $-z$) of length $5$ for which the Salem number is at least Lehmer's number.
Theorem \ref{T-shortness 5} tells us there are no smaller Salem numbers of shortness $5$.

\subsection{Salem polynomials of shortness 6}
For shortness $6$ the situation is more complicated.
There are twelve infinite families of length-$6$ Salem polynomials (all previously known) which alongside Theorem \ref{T-shortness 5} shows that there are infinitely many Salem polynomials (representing infinitely many Salem numbers) that have shortness $6$.
We show that there are finitely many other shortness-$6$ Salem polynomials (Theorem \ref{T-shortness 6}).
We give a complete list for those whose Salem numbers are at least $1.01$. By Dobrowolski's result \cite[Theorem 1]{Dobrowolski2006} there are no Salem polynomials of length 6 whose Salem number is less that $1+10^{-66}$, so extending our search method as close to $1$ as this would give a complete list. However, this is impractical. 
Nevertheless it is plausible that our list is indeed complete. 
Of course, this would be the case if, as is strongly suspected, the conjecture that  $\lambda_0$  is the smallest Salem number is true (a special case of the strong version of `Lehmer's Conjecture' \cite[Section 1.2]{McKeeSmyth2021}).

\begin{theorem}\label{T-shortness 6}
    There are $12$ infinite families of length-6 Salem polynomials, of the shape $P_n(z)=z^nP(z)+\eps P^*(z)$, where $$P(z)\in\{z-2,z^2-z-1,z^3-z-1,z^3-z^2-1,z^4-z^3-1,z^5-z^4-1\}\,,$$ $\eps\in\{-1,1\}$, and $n\ge n_0=n_0(P,\eps)$. 
    For values of $n_0$, see Table \ref{T-Salem families}.
    
    There are finitely many shortness-6 Salem numbers that do not belong to one these infinite families. 
    There are $126$ sporadic length-6 Salem polynomials for which the Salem number is at least $1.01$.
    Of these, 116 correspond to Salem numbers below $\theta_0$, and 10 correspond to Salem numbers above $\theta_0$.

\end{theorem}

\subsection{Cyclotomic factors of polynomials in Salem's families}
Recall that a \emph{Pisot number} is a real algebraic integer $\theta>1$ such that all of the Galois conjugates of $\theta$ (excluding $\theta$ itself) have modulus strictly less than one.
We define a {\it Pisot polynomial} to be  the minimal polynomial of a Pisot number, possibly multiplied by a cyclotomic polynomial.
The smallest Pisot number is $\theta_0 = 1.3247179\dots$, a zero of $z^3-z-1$.

Let $P(z)$ be a Pisot polynomial that is not equal to its reciprocal polynomial.
Salem considered the families of polynomials $z^nP(z) \pm P^*(z)$, to produce sequences of Salem numbers converging to the Pisot zero of $P(z)$, from either side.
We recall a more precise statement of this in Proposition \ref{T-Salemupdated}, adding some contributions of Boyd.
We extract part of that Proposition here for ease of reference, as a technical theorem.

\begin{theorem}\label{T3}
    Let $P(z)$ be an irreducible Pisot polynomial that is not equal to its reciprocal polynomial. 
    Take $\varepsilon=\pm 1$, and define, for $n\in\mathbb{N}$, $P_n(z)=z^nP(z) + \varepsilon P^*(z)$.
    Then 
    \begin{enumerate}
        \item[(i)] there is a finite set of values of $d$ (depending on $P$) for which there exists $n$ such that $\Phi_d(z) \mid P_n(z)$;
        \item[(ii)] if for some integer $n_d$ we have $\Phi_d(z) \mid P_{n_d}(z)$, then $\Phi_d(z) \mid P_n(z)$ if and only if $n\equiv n_d\pmod{d}$.
    \end{enumerate}
\end{theorem}

The finite set of the theorem can be computed, and we do so for the families that appear in Theorem \ref{T-shortness 6}.
Part (i) follows from \cite[Lemma 2.2]{Chinburg1984}, but we give a simpler proof.

\subsection{Salem numbers below \texorpdfstring{$\theta_0$}{theta\_0}}
We produce a table (Table \ref{Ta-AllSalem}) of all known Salem numbers below the smallest Pisot number, $\theta_{0}$.
Rather than presenting minimal polynomials for these numbers, we give one of their short polynomials. 
From these, their minimal polynomials are easily found using  Corollary  \ref{C-cycfactor} below.

\section{Preliminary results needed for our main proofs}\label{S-prelims}

\begin{table}
\begin{center}
\caption{ The Salem polynomials $S(z)$ of shortness $5$ with corresponding  Salem numbers }\label{T-1} 
\begin{tabular}{|c|c|l|}
\hline
 $S(z)$    &     Salem no.        &             \qquad\qquad    $S(z)$ factorised\\ \hline
$z^{12}-z^7-z^6-z^5+1$ & \vphantom{$3^{3^3}$}1.176280818 & $(z^{10}+z^9-z^7-z^6-z^5-z^4-z^3+z+1)\Phi_6(z)$\\ \hline
$z^{14}-z^{11}-z^7-z^3+1$ & \vphantom{$3^{3^3}$}1.176280818 & $(z^{10}+z^9-z^7-z^6-z^5-z^4-z^3+z+1)\Phi_{10}(z)$\\ \hline
$z^{20}-z^{19}-z^{10}-z+1$ & \vphantom{$3^{3^3}$}1.200026524 & $(z^{14}-z^{11}-z^{10}+z^7-z^4-z^3+1)\Phi_6(z)\Phi_{12}(z)$ \\ \hline
$z^{14}-z^{12}-z^7-z^2+1$ & \vphantom{$3^{3^3}$}1.202616744 & $z^{14}-z^{12}-z^7-z^2+1$\\ \hline
$z^{10}-z^6-z^5-z^4+1$ & \vphantom{$3^{3^3}$}1.216391661 & $z^{10}-z^6-z^5-z^4+1$\\ \hline
 $z^{18}-z^{17}-z^9-z+1$ & \vphantom{$3^{3^3}$}1.216391661 & $(z^{10}-z^6-z^5-z^4+1)\Phi_6(z)\Phi_{12}(z)$\\ \hline
$z^{10}-z^7-z^5-z^3+1$ & \vphantom{$3^{3^3}$}1.230391434 & $z^{10}-z^7-z^5-z^3+1$\\ \hline
$z^{16}-z^{15}-z^8-z+1$ & \vphantom{$3^{3^3}$}1.236317932 & $z^{16}-z^{15}-z^8-z+1$\\ \hline
 $z^{10}-z^8-z^5-z^2+1$ & \vphantom{$3^{3^3}$}1.261230961 & $z^{10}-z^8-z^5-z^2+1$\\ \hline
$z^{14}-z^{13}-z^7-z+1$ & \vphantom{$3^{3^3}$}1.261230961 & $(z^{10}-z^8-z^5-z^2+1)\Phi_{10}(z)$\\ \hline
 $z^8-z^5-z^4-z^3+1$ & \vphantom{$3^{3^3}$}1.280638156 & $z^8-z^5-z^4-z^3+1$\\ \hline
$z^{12}-z^{11}-z^6-z+1$ & \vphantom{$3^{3^3}$}1.293485953 & $(z^{10}-z^8-z^7+z^5-z^3-z^2+1)\Phi_6(z)$\\ \hline
 $z^{10}-z^9-z^5-z+1$ & \vphantom{$3^{3^3}$}1.337313210 & $z^{10}-z^9-z^5-z+1$\\ \hline
  $z^6-z^4-z^3-z^2+1$ & \vphantom{$3^{3^3}$}1.401268368 & $z^6-z^4-z^3-z^2+1$\\ \hline
 $z^8-z^7-z^4-z+1$ & \vphantom{$3^{3^3}$}1.401268368 & $(z^6-z^4-z^3-z^2+1)\Phi_6(z)$\\ \hline
 $z^6-z^5-z^3-z+1$ & \vphantom{$3^{3^3}$}1.506135680 & $z^6-z^5-z^3-z+1$\\ \hline
 $z^4-z^3-z^2-z+1$ & \vphantom{$3^{3^3}$}1.722083806 & $z^4-z^3-z^2-z+1$\\

\hline
\end{tabular}

\end{center}

\end{table}

For the proof of Theorem \ref{T-shortness 5}, we need the following.

\begin{lemma}\label{L-2roots}
 Let $a$ and $b$ be fixed positive numbers, and $f(\th):=\sin(a\th)\sin(b\th)$. Then on $[0,2\pi]$ the distinct zeros of $f(\th)$ interlace the zeros of $f'(\th)$.
 \end{lemma}

\begin{proof}
Note that $f'(\theta)/f(\theta) = a\cot(a\theta) + b\cot(b\theta)$.
Given that $a$ and $b$ are positive, we see that $f'(\theta)/f(\theta)$ is strictly decreasing between its poles, and the poles of $f'(\theta)/f(\theta)$ correspond precisely to the zeros of $f(\theta)$.
Hence $f'(\theta)$ has precisely one zero between each consecutive pair of distinct zeros of $f(\theta)$.
\end{proof}

We also need a particular version of the B\'{e}zout identity (or extended Euclidean algorithm) for coprime integers.
\begin{lemma}\label{L-ext}
 Let $a$ and $b$ be fixed coprime positive integers, both at least $3$. Then there is some $\eps=\pm 1$ and positive integers $i<b/2$ and $j<a/2$ such that $ia-jb=\eps$.
 \end{lemma}
\begin{proof} Since $a$ and $b$ are not both even, we can assume that $a$ is odd. From the extended Euclidean algorithm there are integers $i',j': i'a-j'b=1$, so that for any integer
$r$ we have $(i'-rb)a-(j'+ra)b=1$.  We can choose $\eps=\pm 1$ and $r$ such that $i:=(i'-rb)\eps$ lies in the interval $[0,b/2]$. Since $\gcd(i',b)=1$ and $b>2$, in fact $i$ satisfies $0<i<b/2$. Next, put $j:=(j'+ra)\eps$, so that $ia-jb=\eps$, and $j=(ia-\eps)/b$ is $>0$ (since clearly $i$ and $j$ have the same sign). Then
\[
j\le \frac{ia}{b}+\frac1{b}<\frac{a}{2}+\frac1{b} < \frac{a+1}{2}.
\]
Since $a$ is assumed odd, we have 
\[
j\le \frac{a-1}{2}<\frac{a}{2}.
\]
\end{proof}

The following lemma forms the basis of the proof of Proposition \ref{P-min}, and part (viii) of Proposition \ref{T-Salemupdated}. In it, recall that $\omega_n=e^{2\pi i/n}$.

\begin{lemma}[{{
\cite[Section 5.2]{McKeeSmyth2021}}}]\label{L-Smythcyclotomictrick}
For all natural numbers $n$,
\begin{itemize}
\item $-\omega_n$ is a conjugate of $\omega_n$ if and only if $n$ is a multiple of $4$;
\item $-\omega_n^2$ is a conjugate of $\omega_n$ if and only if $n$ is divisible by $2$ but not by $4$;
\item \!\quad$\omega_n^2$ is a conjugate of $\omega_n$ if and only if $n$ is odd.
\end{itemize}
\end{lemma}

\begin{corollary}\label{C-cycfactor} For any Salem polynomial $S(z)$ the minimal polynomial of its corresponding Salem number is 
$S(z)/\gcd(S(z),S(-z)S(z^2)S(-z^2))$.
\end{corollary}
This follows from the lemma because no Salem number $\tau$ is conjugate to $-\tau$, $\tau^2$ or $-\tau^2$.
A similar remark holds for Pisot polynomials,  although in that case one can more simply divide a Pisot polynomial $P$ by $\gcd(P,P^*)$ to remove cyclotomic factors.
In the proof of Theorem \ref{T-shortness 6}, the construction of the infinite sequences of Salem polynomials of length $6$ is by Salem's method \cite{Salem1945}: for a Pisot polynomial $P(z)$, 
a choice of sign $\eps=\pm 1$, and $r\in \mathbb N$ sufficiently large, $z^rP(z)+\eps P^*(z)$ is a Salem polynomial of length $6$.
We will need to be precise about what ``sufficiently large'' means, and we do this in Proposition \ref{T-Salemupdated}, which collates more complete detail about Salem's sequences, nearly all of which is in Boyd's work \cite{Boyd1977}---see the remarks after the statement of the proposition.

\begin{prop}\label{T-Salemupdated}
    Let $P(z)$ be an irreducible Pisot polynomial that is not equal to its reciprocal polynomial, and $Q(z)=P^*(z)$.
    Take $\eps=\pm1$, and define, for $n\in\mathbb N$,
    \[
    P_n(z)=z^nP(z) + \eps Q(z)\,.
    \]
    Let $r=(Q'(1)-P'(1))/P(1)$, and let $\theta$ be the Pisot number that is a zero of $P(z)$.
    Then:
    \begin{enumerate}
        \item[(i)] $P_n(z)$ has at most one zero outside the unit circle.
        If there is one, it is in the real interval $(1,\infty)$.
        \item[(ii)] If $\eps=1$, then $P_n(z)$ has a unique simple zero $\rho_n\in(1,\infty)$, and all other zeros have modulus $\le 1$.
        \item[(iii)]  If $\eps=-1$ and $n<r$, then $P_n(z)$ has all its zeros on the unit circle.
        \item[(iv)] If $\eps=-1$ and $n>r$, then $P_n(z)$ has a unique simple zero $\rho_n\in(1,\infty)$, and all other zeros have modulus $\le 1$.
        \item[(v)] There exists $n_0$ such that $P_n(z)$ is a cyclotomic polynomial if $n<n_0$, and such that $P_n(z)$ has a unique zero $\rho_n$ in $(1,\infty)$ if $n\ge n_0$ (and $\rho_n$ is either a Pisot number or a Salem number).  
        In the latter case, $P_n(z)$ is the  minimal polynomial of $\rho_n$ multiplied by a cyclotomic polynomial.
        \item[(vi)] 
        If $\eps Q(\theta)>0$ then $\{\rho_n\}_{n\ge n_0}$ is a strictly increasing sequence.
        If $\eps Q(\theta)<0$ then $\{\rho_n\}_{n\ge n_0}$ is a strictly decreasing sequence.
        In either case, $\rho_n\to\theta$ as $n\to\infty$.
        \item[(vii)] There is a finite set of $n$ for which $P_n(z)$ has a repeated factor.
        \item[(viii)] There is a finite set of $d$ for which there exists $n$ such that $\Phi_d(z)\mid P_n(z)$.
        If for some $n_d$ we have  $\Phi_d(z) \mid P_{n_d}(z)$, then $\Phi_d(z) \mid P_n(z)$ if and only if $n\equiv n_d \pmod{d}$.
        \item[(ix)] There is a finite set of $n$ for which $\rho_n$ is a Pisot number; for all sufficiently large $n$, $\rho_n$ is a Salem number.
    \end{enumerate}
\end{prop}
Before proving all of this, we make some remarks.
Salem's argument from \cite{Salem1945}  gives (i), (ii), the final sentence in (vi), and (ix).
He also shows that according to the sign of $\eps Q(\theta)$ the $\rho_n$ all lie on one or other side of $\theta$, revealing that there are Salem numbers converging to $\theta$ from both sides.
Boyd \cite{Boyd1977} comments on the monotonicity of Salem's sequences, establishes (iii)--(v), and shows that any square-free Salem polynomial arises via Salem's construction.
One can compute the finite sets mentioned in (vii)--(ix) using the technique described in detail for one case in Section  \ref{S:smallest Pisot limit}.
The hypotheses require that $P$ is an irreducible Pisot polynomial.
More generally, given a Pisot polynomial $P$, one can compute $R=\gcd(P,P^*)$; then $P/R$ is an irreducible Pisot polynomial to which the proposition applies, and the sequence of polynomials produced by $P$ is simply the sequence produced by $P/R$ with all the polynomials in the sequence multiplied by $R$; part (vii) and the final sentence of (viii) might not hold for the sequence of polynomials produced by $P$.
A stronger version of part (vii) is given as Lemma 2.1 in \cite{Chinburg1984}, drawing together some of the remarks in \cite{Boyd1977}.
We need only the weaker statement here, which has an easy proof.
The first half of part (viii) follows from \cite[Lemma 2.2]{Chinburg1984}, but we give an alternative proof that is related to the method we use later for finding the claimed finite set, and is considerably simpler than the proof of Chinburg's lemma.
 Note that our argument for the proof of (ix) also leads to a simpler proof of \cite[Theorem 1]{Chinburg1984}.

\begin{proof}
(i) The argument in \cite{Salem1963} is shorter than that in \cite{Salem1945}, and we summarise it here.
Take any $\delta>0$.
Noting that when $|z|=1$ we have $|P(z)|=|Q(z)|$, we see via a Rouch\'e argument using the circle $|z|=1$ that $(1+\delta)z^nP(z)+\eps Q(z)$ has all but one zero strictly inside the unit circle, and one zero strictly outside.
Since $(1+\delta)P(1)+\eps Q(1)\le\delta P(1)<0$, and using the intermediate value theorem, the unique zero outside the unit circle lies in the interval $(1,\infty)$.
Letting $\delta\to0$, using continuity of zeros as the coefficients vary continuously, gives (i).

(ii) When $\eps=1$, we have $P_n(1)=2P(1)<0$, and the intermediate value theorem on the interval $(1,\infty)$, along with (i), gives (ii).

(iii) When $\eps=-1$, we have $P_n(1)=0$.
We compute $P_n'(1)=P'(1) + nP(1)-Q'(1)$.
If $n<r$, we see that $P_n'(1)>0$, so that $P_n$ is positive for real numbers in the interval $(1,1+\delta)$ for sufficiently small $\delta>0$, and then by (i) and the intermediate value theorem there can be no zero in $(1,\infty)$, or there would have to be at least two (with multiplicity).
Then by (i) and Kronecker's theorem \cite[Theorem 1.3]{McKeeSmyth2021}, we have (iii).

(iv) Contrasting with (iii), we have $P_n'(1)<0$ when $n>r$, so that $P_n$ has a zero in $(1,\infty)$, and we are done by (i).

(v)  The first part follows from (i), (iii) and (iv).
That $\rho_n$ is either Pisot or Salem follows from (i) (since $P_n$ is a reciprocal polynomial, we have $P_n(1/\rho_n)=0$ and all other zeros lie on the unit circle).
The final sentence comes from Kronecker's theorem again.

(vi) Take $n\ge n_0$.
Since $P_n(\theta)=\eps Q(\theta)$, we see that $\rho_n\in(1,\theta)$ if $\eps Q(\theta)>0$, and $\rho_n\in(\theta,\infty)$ if $\eps Q(\theta)<0$.
For $n\ge n_0$ we have, using $\eps Q(\rho_n)=-P(\rho_n)\rho_n^n$, that 
\[
P_{n+1}(\rho_n)=P(\rho_n)\rho_n^{n+1}+\eps Q(\rho_n)=P(\rho_n)\rho_n^{n+1}-P(\rho_n)\rho_n^n=P(\rho_n)\rho_n^n(\rho_n-1)\,.
\]
In the case $\eps Q(\theta)<0$, when $\rho_n>\theta$, this gives $P_{n+1}(\rho_n)>0$, so $\rho_{n+1}<\rho_n$;
and if $\eps Q(\theta)>0$, when $\rho_n<\theta$, we instead get $P_{n+1}(\rho_n)<0$, so $\rho_{n+1}>\rho_n$.

Since $\{\rho_n\}_{n\ge n_0}$ is bounded and monotonic, it converges to some $\theta'$.
If $\theta'$ did not equal $\theta$, then $P_n(\rho_n)=P(\rho_n)\rho_n^n+\eps Q(\rho_n)$ would be unbounded; yet $P_n(\rho_n)=0$ for $n\ge n_0$.
This completes (vi).

(vii) If $P_n(z)$ has a repeated factor, then by (v) it has a repeated zero that has modulus $1$.
Such a zero would be a common zero of $P_n(z)$ and $zP_n'(z)=zP'(z)z^n + nP(z)z^n + \eps z Q'(z)$.
Eliminating $z^n$ between the equations $P_n(z)=0$ and $zP_n'(z)=0$ gives
$Q(z)(zP'(z)+nP(z))-zP(z)Q'(z)=0$.
Noting that $P(z)Q(z)$ does not vanish on the unit circle (here we need $P$ an irreducible Pisot polynomial), this gives $$n\le\sup_{|z|=1}|(P(z)Q'(z)-Q(z)P'(z))/P(z)Q(z)|\,,$$ so that there is finite set 
of possible $n$.
(Note that \cite[Lemma 2.1]{Chinburg1984} is stronger: it states that $P_n(z)$ has no repeated factors unless $\varepsilon=-1$ and $n=\deg(P)-2P'(1)/P(1)$, in which case $(z-1)^3 \mid P_n(z)$ and in that case $P_n(z)/(z-1)^3$ has no repeated factors.)

(viii) After Lemma \ref{L-Smythcyclotomictrick} we see that if $\Phi_d(z)\mid P_n(z)$ then $\Phi_d(z)$ also divides one of $P_n(-z)$, $P_n(z^2)$, or $P_n(-z^2)$.

In the first of these three cases, eliminating $z^n$ between $P_n(z)=0$ and $P_n(-z)=0$ gives $Q(z)P(-z)+(-1)^{n+1}Q(-z)P(z)=0$.
This polynomial cannot be identically zero, else evaluating at $\theta$ would give $0=Q(\theta)P(-\theta)\ne 0$ (here using $P\ne Q$).
Thus there are only finitely many cyclotomic factors in this case.

The second and third cases lead to the vanishing of $Q(z)^2P(z^2)+\eps Q(z^2)P(z)^2$ or $Q(z)^2P(-z^2)+\eps(-1)^nQ(-z^2)P(z)^2$ respectively, and in either case we see by evaluating at $\theta$ that the relevant polynomial is not identically zero, so has finitely many cyclotomic factors.

Now $\Phi_d(z) \mid P_{n_d}(z)$ if and only if $P_{n_d}(\omega_d)=0$, where $\omega_d$ is a primitive $d$th root of unity; since $P_n(\omega_d)$ depends only on $n$ modulo $d$, we see that $\Phi_d(z)\mid P_{n_d+md}(z)$ for all natural numbers $m$.
So if  $\Phi_d(z) \mid P_{n_d}(z)$  and $\Phi_d(z) \mid P_n(z)$, then $\Phi_d(z)$ divides $P_n(z)-P_{n_d}(z) = P(z)(z^n-z^{n_d})$, so this latter polynomial vanishes at $\omega_d$, hence (given $P$ not a multiple of $\Phi_d$) we have $\omega_d^n = \omega_d^{n_0}$ and $n\equiv n_d\pmod{d}$.

(ix) By (vii), for all sufficiently large $n$, if $\Phi_d$ divides $P_n$ then it does so with multiplicity $1$.
Then by (viii), the degree of the product of all such irreducible cyclotomic factors of $P_n$ is bounded.
Since the degree of $P_n$ tends to $\infty$ as $n\to\infty$, and using (v), we are done.

\end{proof}

Applying this in detail to length-$6$ Salem polynomials, we see that we get two infinite families of such polynomials for each length-$3$ Pisot polynomial.
The following lemma shows that there are only finitely many such infinite families.
A {\it trinomial} is an integer polynomial having three coefficients of modulus $1$. A trinomial Pisot polynomial must be of the form $z^n-z^k-1$ with $1\le k<n$ in order that its value at $z=1$ is negative. 
\begin{lemma}\label{L-trinomialPisot}
There are exactly five trinomial Pisot polynomials: $z^2-z-1$, $z^3-z-1$,  $z^3-z^2-1$, $z^4-z^3-1$, and $z^5-z^4-1$.
The first four of these are irreducible, while $z^5-z^4-1=(z^2-z+1)(z^3-z-1)$. There is a further (nontrinomial) Pisot polynomial of length $3$, namely $z-2$.
\end{lemma}
\begin{proof}
 Suppose that $n\ge 6$ and $k<n$. Then $1.3^n>13/3$, and
\[
1.3^n-1.3^k\ge 1.3^n-1.3^{n-1}>(13/3)(1-1/1.3)=1,
\]
so that $z^n-z^k-1$ has a root in $(1,1.3)$. However, by a result of Siegel \cite{Siegel1944}, the smallest Pisot number, the real zero of $z^3-z-1$, is $1.3247\cdots>1.3$. Hence all trinomial Pisot polynomials have degree at most $5$. A straightforward search readily checks that the only ones are those given in the statement of the lemma.
\end{proof}

No Salem polynomial can have length 3 (any reciprocal length-$3$ polynomial is cyclotomic), but we may ask if any Salem number can ever be a zero of a trinomial.
The answer is `no', and we record this as a remark.

\begin{remark}\label{L-trinomialSalem}
    Let $\tau$ be a Salem number, and let $P(z)$ be a trinomial $z^n \pm z^m \pm 1$ with $n>m\ge 0$.
    Then $P(\tau)\ne0$.
\end{remark}

\begin{proof}
Suppose that $P(\tau)=0$.    
Thus $\tau^n=\varepsilon_1\tau^m+\varepsilon_2$ for some $\varepsilon_1$, $\varepsilon_2\in\{-1,1\}$.
    Now let $\beta$ be a conjugate of $\tau$ on the unit circle.
    Then $\beta^n = \varepsilon_1\beta^m + \epsilon_2$.
    Thus $\beta^n$ lies on the intersection of the two circles $|z|=1$ and $|z-\varepsilon_2|=1$.
    This gives $\beta^n$ a root of unity (order either $3$ or $6$), so that $\beta$ is a root of unity.
    This contradicts $\beta$ being a conjugate of $\tau$.
\end{proof}

\section{Proof of Theorem \ref{T-shortness 5}  }\label{S:P1}
First note that a Salem polynomial $S(z)$  of length $5$ cannot be antireciprocal, and the coefficients  of its terms of highest and lowest degree must be $1$, and the other three coefficients $-1$. For if it were antireciprocal, its value at $1$ would be $0$, and so it would have even length. Thus $S$ is reciprocal. Since it is monic, the coefficients of highest and lowest degree terms are $1$. Since $S(1)<0$, the other three coefficients must be $-1$
(note that $z^{2n}-3z^n+1$ has $n$ zeros outside the unit circle; if $n=1$ then it is a Pisot polynomial, and for no $n$ do we get a Salem polynomial).

We can now put  $S(z)=z^{2n}-z^{2n-k}-z^n-z^k+1$, where $1\le k<n$. To study $S(z)$ on the unit circle, we put $z=e^{i\th}$, and define $s(\th)$ so that 
\be\label{E-sth}
s(\th):=z^{-n}S(z)+1=2\cos(n\th)-2\cos((n-k)\th) = -4\sin\left(\frac{2n-k}{2}\th\right)\sin\left(\frac{k}{2}\th\right).
\ee
We need to find all $n$ and $k$ for which the equation $s(\th)=1$ has $2n-2$ roots in $(0,2\pi)$. First note that we must have $\gcd(n,k)=1$, as otherwise $S(z)$ would have at least  $\gcd(n,k)>1$ zeros outside the unit circle. We divide the argument into two parts, as follows.
\begin{itemize}
\item[(a)] The case $k$ even, $=2k'$ say. Therefore $n$ is odd and $\gcd(n-k',k')=\gcd(n,k')=1$. Here 
\[
s(\th)=-4\sin((n-k')\th)\sin(k'\th),
\]
which has zeros in $[0,2\pi)$ at $\ell\pi/(n-k')\,\,(\ell=0,\dots,2(n-k')-1)$ and $m\pi/k'\,\,(m=0,\dots,2k'-1)$. Because of its double 
zeros at $\theta=0$ and $\th=\pi$, the function $s(\th)$ has $2+((n-k')-1)+(k'-1)=n$ distinct zeros on $[0,\pi]$. Now pick particular $\ell$ and $m$ such that 
\be\label{E-ext}
\ell k'-m(n-k')=1.
\ee
If $k'=1$ we can take $m=0$ and $\ell=1$. For $k'\ge 2$, from
\[
(\ell+j(n-k'))k'-(m+jk')(n-k')=1
\]
we see that we can choose $0<m<k'$, from which it follows easily that $0<\ell<n-k'$. Thus for $k'\ge 1$ we have
\[
\frac{\ell}{n-k'}-\frac{m}{k'}=\frac1{k'(n-k')},
\]
so that the interval $I:=[\pi m/k',\pi\ell/(n-k')]$ has endpoints that are consecutive zeros of $s(\th)$. Taking $\th=\pi m/k'+\pi\la/(k'(n-k'))\,\,(0\le\la\le 1)$ in this interval, we have
\begin{equation}\label{E-s1}
\sin(k'\th)=(-1)^m\sin\left(\frac{\la\pi}{n-k'}\right) \quad\text{ and }\quad \sin((n-k')\th)=(-1)^{\ell+1}\sin\left(\frac{(1-\la)\pi}{k'}\right),
\end{equation}
 the second equality using \eqref{E-ext}. Hence, using $\sin x\le x$ for $x>0$ we obtain
\[
|s(\th)|\le 4\frac{\la(1-\la)\pi^2}{k'(n-k')}<1 \quad\text{ for }\quad k'(n-k')>\pi^2.
\]
This latter inequality holds for $n>\pi^2+1$, i.e., $n\ge 11$.

Now,  $[0,\pi]$ can be partitioned into $n-1$ intervals where, for each such interval,   $s(\th)$ is $0$ at both  endpoints, and of one sign in between. Since $s(\th)<0$ for small positive $\th$, the first such interval has a local minimum, and, for subsequent intervals,
maxima and minima alternate. 
Thus there are $\lfloor (n-1)/2\rfloor$ intervals 
 with maxima. If $s(\th)>1$ at such a maximum, then, by  Lemma \ref{L-2roots}, that interval will contain exactly two values of $\th$ for which $s(\th)=1$. If all such intervals have maximum $>1$, then $s(\th)=1$ will have $2\lfloor (n-1)/2\rfloor$ solutions on $[0,\pi]$, and so, because $n$ is odd, a total of $4\lfloor (n-1)/2\rfloor=2n-2$ solutions on $[0,2\pi]$. Thus $S(z)$ will be a Salem polynomial, its other two zeros being (say) $\tau>1$ and $1/\tau$.

 For $n\ge 11$ , however, $s(\th)$ has, from \eqref{E-s1}, maximum modulus $<1$ on at least one of the small subintervals above. In fact, because $s(\pi-\th)=-s(\th)$, there are at least two such intervals, on one of which $s(\th)$ has a maximum $<1$. Hence, $S(z)$ is not a Salem polynomial for $n\ge 11$.
 
 \item[(b)] The case $k$ odd. Then the function $s(\th)$ given by \eqref{E-sth}
 is even, of period $2\pi$. Its zeros in $[0,2\pi)$ are at 
 \[
\th=\frac{2\ell\pi}{2n-k}\quad(\ell=0\cc 2n-k-1)\,\,\text{ and } \th=\frac{2m\pi}{k}\quad(m=0\cc k-1),
 \]
 with $0$ being the only repeated zero. Thus $s(\th)$ has $2n-1$ distinct zeros on $[0,2\pi)$,
 and so, as in case (a), $[0,2\pi]$ can be divided into $2n-1$ subintervals, each having two consecutive zeros of $s(\th)$ as endpoints. On each subinterval, $s(\th)$ alternates between a minimum and a maximum, with minima on the first and last intervals.

 For the case $k=1$, we note that the second subinterval is $[2\pi/(2n-1),4\pi/(2n-1)]$,
 which has a local maximum. On this subinterval, $\th=2\pi(1+\la)/(2n-1)$, for $0\le\la\le 1$,
 so that here
 \[
s(\th)=-4\sin(\pi(1+\la))\sin(\pi(1+\la)/(2n-1)),
 \]
from which $|s(\th)|\le 8\pi/(2n-1)$, which is less than $1$ for $n\ge 14$.

 Now assume that $k\ge 3$. Then also $2n-k=n+(n-k)\ge 3$,  and we can use Lemma \ref{L-ext} to choose $\ell$ and $m$ such that 
 \be\label{E-ext3}
 \ell k-m(2n-k) =\eps=\pm 1,
 \ee
 where $\ell<(2n-k)/2$ and $m<k/2$. Thus $2\pi\ell/(2n-k)$ and $2\pi m/k$ (in some order) are the endpoints of a subinterval, call it $I$, of length $2\pi/(k(2n-k))$. In a similar manner to (a) we calculate that, at $\la 2\pi\ell/(2n-k)+(1-\la)2\pi m/k\in I$,
 \begin{align}
 |s(\th)|&=\left|4\sin\left(\pi\left(\ell-(1-\la)\frac{\eps}{k}\right)\right)\sin\left(\pi\left(m+\la\frac{\eps}{2n-k}\right)\right)\right| \label{E-bd1}\\
 &\le 4\frac{\la(1-\la)
 \pi^2}{k(2n-k)}\le \frac{\pi^2}{k(2n-k)}\le \frac{\pi^2}{3(2n-3)}.\label{E-bd2}
 \end{align}
 By counting the number of zeros of $s(\th)$ before the two that are the endpoints of $I$,
 we see that $\ell+m$ must be even for $I$ to have a maximum. 
 However, if $\ell+m$ is odd, we can replace $\ell$ and $m$ by $2\ell$ and $2m$. Not only do these have an even sum, but $4\pi\ell/(2n-k)$ and $4\pi m/k$ (in some order) are also consecutive zeros of $s(\th)$. This follows from the fact that $2\ell<k$ and $2m<2n-k$, combined with the fact that the interval, call it $I_2$, having these zeros as endpoints, is too short -- its length is $4\pi/(k(2n-k))$ --  to contain a  subinterval of length  $2\pi/(2n-k)$ or $2\pi/k$.

 When we replace $\ell,m$ by $2\ell,2m$ in \eqref{E-ext3} to give $2\ell k-2m(2n-k) =2\eps$, we see that we must replace $\eps$ in \eqref{E-bd1} by $2\eps$, and hence $\pi^2$ in \eqref{E-bd2} by $4\pi^2$ (in 3 places). Thus from \eqref{E-bd1},\eqref{E-bd2},  modified in this way, we see that $|s(\th)|<1$ on $I_2$ for $n\ge 9$. Thus all Salem polynomials with $k\ge 3$ and odd must have $n\le 8$.

 We have shown that $S(z)$ is not a Salem polynomial if $n\ge 14$. A routine search for $4\le n\le 13$ with $k<n$ and $\gcd(k,n)=1$ readily gives all Salem polynomials of shortness $5$, as shown in Table \ref{T-1}.
\end{itemize}

\section{Proof of Theorem \ref{T-shortness 6}: \texorpdfstring{\\}{\ } Short Salem polynomials in intervals}
\label{S:intervalsearch}

We have in mind the desire to prove Theorem \ref{T-shortness 6}, but for the moment we consider the more general question of finding short Salem numbers in intervals.

If $P(z)=a_0z^d + \cdots + a_d\in\mathbb{R}[x]$ has Mahler measure $M$, then Gon\c  calves' inequality \cite[Proposition 1.12]{McKeeSmyth2021} states that $M^2 + a_0^2a_d^2/M^2\le a_0^2 + \cdots + a_d^2$.
Suppose further that $P$ is a monic reciprocal or antireciprocal  polynomial of length $\ell\ge 5$.
Then the sum of the squares of its coefficients is at most $1^2+(\ell-2)^2+(\pm 1)^2=\ell^2-4\ell+6$. 
In particular, if $P$ is the short polynomial of a Salem number $\tau$ then we get that $\tau^2+1/\tau^2\le \ell^2-4\ell+6$ and hence a bound $b=b(\ell)$  such that $\tau\le b$. 
In the other direction, Dobrowolski's lower bound on the Mahler measure of $P$ as a function of its number of terms \cite{Dobrowolski1991, Dobrowolski2006} shows that there is some $a>1$ (again depending on $\ell$, but otherwise independent of $\tau$ and $P$) such that $\tau\ge a$. 
Thus $\tau$ will lie in the interval $[a,b]$.

We now consider how we might attempt to find all length-$\ell$ Salem polynomials in an interval $[a,b]$, with $1<a<b$.
Choosing $a$ and $b$ as above, this would give all length-$\ell$ Salem polynomials, but in practice we might restrict to a shorter interval.
There might be infinite families (those of Salem, or others), and we will mix theory with pragmatism in situations where we cannot with certainty give a complete description: what is the best that we can say after a reasonable amount of computation?

A key tool is the algorithm described in \cite[Section 1.7.1]{McKeeSmyth2021}, adapted to the search for reciprocal polynomials.
The technique we use is similar to that in \cite{El-SerafyMcKee2025}, but with the restriction to Salem polynomials allowing further efficiencies.
We are also able to cope with certain limit points in the interval, including all those that arise when the length is $6$: we can identify families approaching those limit points, and otherwise restrict to finitely many `sporadic' examples.
We are looking for Salem numbers in $[a,b]$ that are zeros of Salem polynomials of length $\ell\ge 6$ (the length-$5$ case having been settled in Theorem \ref{T-shortness 5}).
We write our length-$\ell$ Salem polynomial $P(z)$ as
\[
z^n + \eps_1 z^{n-g_1} + \eps_2 z^{n-g_1-g_2} + \cdots +\eps_{\ell-1}\,,
\]
where the first ``gap'' $g_1$ is strictly positive, but later gaps $g_i$ are only weakly positive; each $\eps_i$ is either $1$ or $-1$, and there is no cancellation of terms.
Since $P(z)=\pm P^*(z)$, if we know at least the first half of the $g_i$ and $\eps_i$, then we know the complete polynomial.

The bounds in \cite[Section 1.7.1]{McKeeSmyth2021}, given that our Salem number is at least $a$, give
\[
g_1 \le \lfloor \log(\ell - 1)/\log a \rfloor\,,
\]
\[
g_2 \le \lfloor (\log(\ell-2)-\log|a^{g_1}-1|)/\log a\rfloor\,,
\]
and more generally, 
\[
g_i\le\lfloor (\log(\ell - i)-\inf_{x\in[a,b]}\log|Q(x)|)/\log a\rfloor\,,
\]
where
\[
Q(z)=z^{g_1+\dots+g_{i-1}}+\eps_1z^{g_2+\dots+g_{i-1}}+\cdots + \eps_{i-1}\,.
\]
We can always bound $g_1$ and $g_2$ by this approach, but we may already hit trouble with $g_3$, if $Q(z)$ has a zero in the real interval $[a,b]$.

If $Q(z)$ has a zero in $[a,b]$, what can we do?
If $Q(z)$ has more than one zero outside the unit circle, then we can apply a Rouch\'e argument to bound the next gap, in a manner similar to that used in \cite{El-SerafyMcKee2025}: encircle the zeros of $Q$ outside the unit circle using small enough radii that the circles produced do not overlap, and do not meet the unit circle; compute rigorous lower bounds for $|Q|$ on all these circles; if the next gap is too large, then a Rouch\'e estimate shows that there will be at least two zeros outside the unit circle, which excludes all Salem polynomials.

The only situation where we are unable to bound the next gap is if $Q(z)$ has a unique zero outside the unit circle, lying in the interval $[a,b]$.
This troublesome zero is therefore either a Salem number or a Pisot number.
If $\ell$ is even and $Q(z)$ has length $\ell/2$, then we are saved: either we have identified one of Salem's families, or we have found an infinite sequence of Salem polynomials of the shape $Q(z)(z^n \pm 1)$ where $Q(z)$ is itself a Salem polynomial.

If $\ell=6$, this approach succeeds in identifying all Salem numbers in $[a,b]$ that are represented by length-$6$ Salem polynomials.
Together with Lemma \ref{L-trinomialPisot}, and taking $1<a<b$ as above so that all length-$6$ Salem numbers have their corresponding Salem number in the interval $[a,b]$, this establishes most of Theorem \ref{T-shortness 6}: there are $12$ infinite families of length-$6$ Salem polynomials, and otherwise finitely many length-$6$ Salem numbers.
The existence of a transition point $n_0=n_0(P)$ such that $P_n(z)$ in Theorem \ref{T-shortness 6} is a Salem polynomial if and only if $n\ge n_0$ follows from the computations in Section \ref{S:smallest Pisot limit}: \textit{a priori} one could imagine the possibility that a quadratic Pisot number appears after the first Salem number in the sequence, but for the families in Theorem \ref{T-shortness 6} this is not so.

If $\ell>6$, the situation is muddied if our interval $[a,b]$ contains any Salem or Pisot number of shortness below $\ell/2$.
We can record such obstacles and move on, but unless we find an \emph{ad hoc} way to deal with these afterwards, our search will be incomplete.
We may also hit practical issues if the finite search is unreasonably large.
In such cases we can impose bounds on the gaps to restore a practical search, but then need to record whenever those bounds are enforced, as if that happens then the search is incomplete.

To complete the proof of Theorem \ref{T-shortness 6}, we took $\ell=6$, $a=1.01$, and $b=\sqrt{10}$ (slightly larger than needed).
In fact we first took $a=1.17$ and within a few seconds computed the $126$ polynomials shown in Table \ref{Ta-Salem6}.
Then, as a background task over many weeks, we gradually pushed the lower bound down to $1.01$ (of course finding no new examples).
Note that Table \ref{Ta-Salem6} refers to `sporadic Salem polynomials'.
This simply excludes all those polynomials that appear in one of Salem's families.
On the other hand, Salem numbers in that table might \emph{also} be zeros of Salem polynomials that appear in one or more of the infinite families: these are indicated by a subscript $f$.
Those that are actually of shortness $5$ are given a subscript $5$.

\setlength{\LTcapwidth}{0.9\textwidth}
\newpage

\begin{longtable}[H]{|l|l|}
\caption{\\ All 126 known sporadic Salem polynomials of length 6. Here $[n,k,d,\eps]$ denotes the polynomial $z^n(z^k-z^{k-d}-1)+\eps(z^k+z^d-1)$.  A subscript $f$ indicates that the Salem number also appears in one or more of the twelve infinite families; a subscript $5$ indicates that the Salem number has shortness $5$.} \label{Ta-Salem6}\\   \hline
Salem number &  \qquad\qquad Corresponding polynomials of length 6     \\ \hline \hline
\endfirsthead
\caption{(continued)}\\ \hline
 Salem number &  \qquad\qquad Corresponding Salem polynomials of length 6     \\ \hline \hline
\endhead
\endfoot
\endlastfoot
$1.176280818_{f,5}$ & $\vphantom{3^{3^3}} [25, 11, 1, 1], [22, 13, 1, -1], [12, 5, 2, 1], 
[20, 7, 2, 1], [14, 11, 2, -1], $\\
 & $\vphantom{3^{3^3}}[14, 4, 3, 1], [19, 5, 3, 1], [18, 7, 3, -1], [11, 10, 3, -1], [13, 7, 4, -1],$ \\
 & $\vphantom{3^{3^3}}[13, 6, 5, -1], [11, 7, 5, -1], [10, 7, 6, -1]$ \\ \hline
 1.188368148 & $\vphantom{3^{3^3}}[23, 10, 1, 1], [20, 12, 1, -1], [13, 10, 2, -1], 
  [13, 6, 4, -1]$ \\ \hline
$1.200026524_5$ & $\vphantom{3^{3^3}}[20, 9, 1, 1], [29, 10, 1, -1], [19, 11, 1, -1], 
  [11, 10, 2, -1], [17, 4, 3, 1],$ \\
  & $\vphantom{3^{3^3}}[23, 5, 3, -1], [10, 8, 3, -1], [11, 6, 4, -1]$ \\ \hline
 $1.202616744_5$ & $\vphantom{3^{3^3}} [21, 9, 1, 1], [18, 11, 1, -1], [14, 5, 2, 1], 
    [19, 7, 2, -1], [12, 9, 2, -1],$ \\
    & $\vphantom{3^{3^3}}[9, 7, 4, -1]$ \\ \hline
$1.216391661_{f,5}$ &  $\vphantom{3^{3^3}}[18, 8, 1, 1], [26, 9, 1, -1], [17, 10, 1, -1], 
  [16, 5, 2, 1], [10, 9, 2, -1], $ \\
 &   $\vphantom{3^{3^3}} [24, 4, 3, 1], [15, 5, 3, -1],  [11, 5, 4, -1], [9, 6, 4, -1], [8, 6, 5, -1] $ \\ \hline
 1.219720859 &  $\vphantom{3^{3^3}}[19, 8, 1, 1], [16, 10, 1, -1], [11, 8, 2, -1], 
   [31, 4, 3, 1] $ \\ \hline
 $1.230391434_{f,5}$ & $\vphantom{3^{3^3}} [15, 7, 1, 1], [28, 8, 1, 1], [21, 5, 2, 1], 
   [19, 4, 3, -1], [12, 5, 3, -1], $ \\
   & $\vphantom{3^{3^3}}[8, 7, 3, -1], [7, 6, 5, -1]$ \\ \hline
  1.232613549 &  $\vphantom{3^{3^3}}[33, 8, 1, -1], [16, 9, 1, -1], [23, 5, 2, 1]$ \\ \hline
                  1.235664580 & $\vphantom{3^{3^3}} [30, 5, 2, 1]$ \\ \hline
 $1.236317932_5$ & $\vphantom{3^{3^3}} [16, 7, 1, 1], [23, 8, 1, -1], [15, 9, 1, -1], 
   [37, 5, 2, 1], [9, 8, 2, -1]$ \\ \hline
                 1.237504821 & $\vphantom{3^{3^3}}[29, 5, 2, -1] $ \\ \hline
1.240726424 & $\vphantom{3^{3^3}} [12, 6, 1, 1], [17, 7, 1, 1], [14, 9, 1, -1], 
  [22, 5, 2, -1], [10, 7, 2, -1], $ \\
   & $\vphantom{3^{3^3}} [15, 4, 3, -1], [7, 6, 4, -1] $ \\ \hline
          1.252775937 &  $\vphantom{3^{3^3}} [13, 6, 1, 1], [24, 7, 1, 1] $ \\ \hline
           1.253330650 &  $\vphantom{3^{3^3}} [25, 7, 1, 1], [15, 5, 2, -1]$ \\ \hline
           1.255093517 & $\vphantom{3^{3^3}} [33, 7, 1, 1], [12, 4, 3, -1]$ \\ \hline
          1.256221154 & $\vphantom{3^{3^3}} [29, 7, 1, -1], [14, 8, 1, -1]$ \\ \hline
                  1.260103540 & $\vphantom{3^{3^3}} [21, 7, 1, -1]$ \\ \hline
 $1.261230961_{f,5}$ & $\vphantom{3^{3^3}} [14, 6, 1, 1], [20, 7, 1, -1], [13, 8, 1, -1], 
    [13, 5, 2, -1], [8, 7, 2, -1], $ \\
     & $\vphantom{3^{3^3}} [7, 5, 4, -1] $ \\ \hline
  1.267296443 & $\vphantom{3^{3^3}} [15, 6, 1, 1], [12, 8, 1, -1], [9, 6, 2, -1]$ \\ \hline
  $1.280638156_{f,5}$ & $\vphantom{3^{3^3}} [20, 6, 1, 1], [13, 7, 1, -1], [10, 5, 2, -1], 
   [9, 4, 3, -1], [7, 5, 3, -1], $ \\
    & $\vphantom{3^{3^3}} [6, 5, 4, -1]$ \\ \hline
                  1.281691372 & $\vphantom{3^{3^3}} [21, 6, 1, 1]$ \\ \hline
                  1.282495561 & $\vphantom{3^{3^3}} [22, 6, 1, 1]$ \\ \hline
                  1.284616551 & $\vphantom{3^{3^3}} [28, 6, 1, 1]$ \\ \hline
                  1.284746822 & $\vphantom{3^{3^3}}  [29, 6, 1, 1]$ \\ \hline
                  1.285099364 & $\vphantom{3^{3^3}}  [35, 6, 1, 1]$ \\ \hline
                   1.285121520 & $\vphantom{3^{3^3}}  [36, 6, 1, 1]$ \\ \hline
                  1.285185671 & $\vphantom{3^{3^3}}  [43, 6, 1, 1]$ \\ \hline
           1.285196727 & $\vphantom{3^{3^3}}  [50, 6, 1, 1]$
           \\ \hline 
           1.285199179 & $\vphantom{3^{3^3}}[61, 6, 1, -1]$ \\ \hline
                  1.285235436 & $\vphantom{3^{3^3}}  [39, 6, 1, -1]$ \\ \hline
                 1.285409065 & $\vphantom{3^{3^3}}  [32, 6, 1, -1]$ \\ \hline
         1.286395967 & $\vphantom{3^{3^3}}  [25, 6, 1, -1], [12, 7, 1, -1]$ \\ \hline
                 1.286730182 & $\vphantom{3^{3^3}}  [24, 6, 1, -1]$ \\ \hline
                 1.291741426 & $\vphantom{3^{3^3}}  [18, 6, 1, -1]$ \\ \hline
 $1.293485953_{f,5}$ & $\vphantom{3^{3^3}}  [17, 6, 1, -1], [11, 7, 1, -1], [7, 6, 2, -1], 
    [8, 4, 3, -1]$ \\ \hline
                1.295675372 & $\vphantom{3^{3^3}}  [16, 6, 1, -1]$ \\ \hline
          $1.302268805_f$ & $\vphantom{3^{3^3}}  [10, 7, 1, -1], [8, 5, 2, -1]$ \\ \hline
                 $1.318197504_f$ & $\vphantom{3^{3^3}}  [11, 6, 1, -1]$ \\ \hline
                  $1.326633115_f$ & $\vphantom{3^{3^3}}  [10, 6, 1, -1]$ \\ \hline
          $1.337313210_{f,5}$ & $\vphantom{3^{3^3}}  [9, 6, 1, -1], [6, 5, 2, -1]$ \\ \hline
          $1.350980338_f$ & $\vphantom{3^{3^3}}  [8, 6, 1, -1], [7, 4, 2, -1]$ \\ \hline
                  $1.359999712_f$ & $\vphantom{3^{3^3}}  [5, 4, 3, -1]$ \\ \hline
          $1.401268368_{f,5}$ & $\vphantom{3^{3^3}}  [5, 4, 2, -1], [4, 4, 3, -1]$ \\ \hline
                  $1.556030191_f$ & $ \vphantom{3^{3^3}}  z^5(z^2-2) -2z^2+1 =[5,2,2,-1] $ \\ \hline               $1.722083806_{f,5}$ & $ \vphantom{3^{3^3}} z^3(z^2-2) -2z^2+1 = [3,2,2,-1]$
\\ \hline
\end{longtable}

\section{The minimal polynomials of Salem numbers in Salem's length-6 families}\label{S:smallest Pisot limit}

In this section, working through one of the most interesting examples, we indicate how the various finiteness assertions in Proposition \ref{T-Salemupdated} can be made explicit.
Then in Table \ref{T-Harriett} we summarise the corresponding data for all the infinite families of length-6 Salem polynomials.

Taking $P(z)=z^3-z-1$ and $\eps=-1$ in Proposition \ref{T-Salemupdated}, we define $P_n$ and $\tau_n$ accordingly.
The value of $r$ in the statement of the theorem is $7$, and we then check that $n_0$ in part (v) is $8$.
We therefore know that $P_n(x)$ has, for $n\ge 8$ exactly one root outside the unit circle, so that the non-cyclotomic part of $P_n(x)$ is irreducible. 
Thus to find the degree of this non-cyclotomic part we just need to find all the cyclotomic factors of $P_n(z)$, for all $n>7$, and subtract their degrees from the degree of $P_n(z)$.

To do this, we first need to define 
\[
\delta_{j,k}(n):=\begin{cases} 1 &\text{ if $n\equiv j\pmod{k}$};\\
0 &\text{ otherwise}.\end{cases}
\]
 We can then define the function $F(n)$ by
\begin{equation}\label{E-deg}
F(n):=\delta_{1,{2}}(n)+2\delta_{1,{3}}(n)+4\delta_{2,{5}}(n)+4\delta_{3,{8}}(n)+4\delta_{4,{12}}(n)\\
+6\delta_{5,{18}}(n)+8
\delta_{6,{30}}(n).
\end{equation}
Note that the function $F(n)$ is periodic of period equal to the least common multiple 
\[
\lcm(2,3,5,8,12,18,30)=360.
\]
Its maximum value is $15$, attained for $n=347$. 

We claim the following.

\begin{prop} \label{P-min}The minimal polynomial $S_n(x)$ of $\tau_n$ for $n\ge 8$ is
given by
\begin{equation}\label{E-min}
S_n(z)=\frac{P_n(z)}{
\Phi_1(z)\Phi_2(z)^{\delta_{1,2}}\Phi_3(z)^{\delta_{1,3}}\Phi_5(z)^{\delta_{2,5}}\Phi_8(z)^{\delta_{3,8}}\Phi_{12}(z)^{\delta_{4,12}}\Phi_{18}(z)^{\delta_{5,18}}\Phi_{30}(z)^{\delta_{6,30}}},
\end{equation}
and has degree $f_n:=n+2-F(n)$.
\end{prop}

One can readily compute that $\tau_n$ has degree $>50$ for all $n\ge 60$, with $\tau_{59}$ having degree $50$.

In the proof, we make use of the resultant of two nonzero polynomials $A,B \in R[z]$, where $R$ is an integral domain. The resultant $\Res_z(A,B)\in R$, the result of eliminating $z$ from $A$ and $B$, is $0$ if and only if $A$ and $B$ have a common nonconstant factor in $R[z]$.

\begin{proof} Let $\phi(n)$ be the degree of the $n$th cyclotomic polynomial $\Phi_n(z)$. We immediately see that $P_n(z)$ is divisible by $\Phi_1(z)=z-1$ for all $n$.
The idea of the proof is to show that any polynomial $P_n(z)$ can be divisible by $\Phi_k(z)$ only for $k=1,2,3,5,8,12,18$ or $30$, and that, to be precise, 
\begin{itemize}
\item $\Phi_2(z)$\,\, for $n\equiv 1\pmod{2}$;
\item $\Phi_3(z)$\,\, for $n\equiv 1\pmod{3}$;
\item $\Phi_5(z)$\,\, for $n\equiv 2\pmod{5}$;
\item $\Phi_8(z)$\,\, for $n\equiv 3\pmod{8}$;
\item $\Phi_{12}(z)$ for $n\equiv 4\pmod{12}$;
\item $\Phi_{18}(z)$ for $n\equiv 5\pmod{18}$;
\item $\Phi_{30}(z)$ for $n\equiv 6\pmod{30}$.
\end{itemize}
Since for $n\ge 8$ the polynomial $P_n(z)/(z-1)$ has no repeated roots (\cite{Boyd1977}, \cite[Lemma 2.1]{Chinburg1984}),
has degree $n+2$,  and $\Phi_k(z)$ has degree $\phi(k)$, the result will follow.
 We now give the details. Our method was used earlier in \cite{Smyth2000}.
 Define $T_3(x,y):=y(x^3-x-1)+x^3+x^2-1$, so that $P_n(x)=T_3(x,x^n)$. Put $y=x^n$. Then the map 
\begin{itemize} 
\item
$x \mapsto -x$ takes 
\[
\begin{cases} y \mapsto -y &\text{ for $n$ odd};\\
y \mapsto y &\text{ for $n$ even};\end{cases} 
\]
\item $x \mapsto -x^2$ takes 
\[
\begin{cases} y \mapsto -y^2 &\text{ for $n$ odd};\\
y \mapsto y^2 &\text{ for $n$ even};\end{cases}
\]
\item $x \mapsto x^2$ takes $y\mapsto y^2$.
\end{itemize}

By Lemma \ref{L-Smythcyclotomictrick}, every cyclotomic factor of $P_n(z)$ must also be a factor of  one of $P_n(-z)$, $P_n(-z^2)$ or $P_n(z^2)$.  In terms of $T_3(x,y)$ this implies that the only possible such cyclotomic factors are the cyclotomic factors of 
\begin{align*}
&\Res_y(T_3(x,y),T_3(-x,-y)) &&= \Phi_1(x)\Phi_2(x)\Phi_8(x) 
&& \text{ ($n$ odd)};\\
&\Res_y(T_3(x,y),T_3(-x,y))&&=x\Phi_{12}(x) 
&& \text{ ($n$ even)};\\
&\Res_y(T_3(x,y),T_3(-x^2,-y^2))&&=x\Phi_1(x)\Phi_2(x)^3\Phi_{18}(x)
&& \text{ ($n$ odd)};\\
&\Res_y(T_3(x,y),T_3(-x^2,y^2))&&= \Phi_{12}(x)\Phi_{30}(x)
&& \text{ ($n$ even)};\\
&\Res_y(T_3(x,y),T_3(x^2,y^2))&&=\Phi_1(x)^3\Phi_2(x)\Phi_3(x)^2\Phi_5(x).\, &&
\end{align*}
Hence we see that the only possible cyclotomic factors of $P_n(z)$ are
$\Phi_1(z)=z-1$, $\Phi_2(z)= z+1$, $\Phi_3(z)= z^2+z+1$, $\Phi_5(z)= z^4 + z^3 + z^2 + z + 1$, $\Phi_8(z)= z^4+1$, $\Phi_{12}(z)= z^4 - z^2 + 1$, $\Phi_{18}(z)= z^6 - z^3 + 1$ and $\Phi_{30}(z)= z^8 + z^7 - z^5 - z^4 - z^3 + z + 1$. 
Now for $k=2,3,5,8,12,18,30$ we easily check that the least values $n_k$  for which $\Phi_k(z)\mid P_{n_k}(z)$ are $1,1,2,3,4,5,6$, respectively. 
Hence, by Theorem \ref{T3}(ii), $\Phi_k(z)\mid P_{n}(z)$ precisely for $n\equiv n_k\pmod{k}$, as claimed at the start of the proof.

Finally, since for these values of $k$ the cyclotomic polynomials $\Phi_k(z)$ have degrees $1,2,4,4,4,6,6$ respectively, and because for $n\ge 8$ the polynomial $P_n(z)$ has no repeated roots, we obtain the formula for $S_n(z)$ as stated in  \eqref{E-min}, as well as the formula \eqref{E-deg} for $F(n)$. Then since the numerator of $S_n(z)$ has degree $n+3$ and its denominator has degree $1+F(n)$, we obtain $f_n=n+2-F(n)$ as claimed.

\end{proof}

Similar computations can be done for all of the twelve infinite families of length-$6$ Salem polynomials.\footnote{These computations were done by Harriett Du Four as part of an undergraduate project at Royal Holloway, University of London.}
Writing the general length-$6$ Salem polynomial as in Theorem 3.1 (but allowing the possibility $P(z)=z^5-z^4-1=(z^2-z+1)(z^3-z-1)$) in the shape $P_n(z) = z^nP(z) + \varepsilon P^*(z)$, where $\varepsilon\in\{-1,1\}$ and $P(z)$ is one of the six trinomials listed in Lemma \ref{L-trinomialPisot}, it was found that 
the only appearances of reciprocal quadratic Pisot polynomials as factors of $P_n$ were for $\varepsilon = 1$ and $(P,n)\in\{(z-2,1),(z-2,2),(z^2-z-1,1)\}$.
Hence for each choice of $P$ and $\varepsilon$ there is some $n_0$ such that $P_n$ is a Salem polynomial if and only if $n\ge n_0$.
In Table \ref{T-Salem families}, we show the value of $n_0$, and also indicate all the cyclotomic factors of any of the $P_n$.
The case $P(z)=z^5-z^4-1$ is excluded, as it mirrors $P(z)=z^3-z-1$ with an extra cyclotomic factor of $z^2-z+1$ for all $n$.

\begin{table}\label{T-Harriett}
\caption{Cyclotomic factors of the length-$6$ Salem polynomials that arise from Salem's construction.  The $n$th polynomial in the family is $z^nP(z) + \varepsilon P^*(z)$, and this polynomial is Salem for $n\ge n_0$.}\label{T-Salem families}
\begin{tabular}{|r|r|c|l|} \hline
  $P$   & $\varepsilon$ &  $n_0$ & list of pairs $(a,d)$ such that $\Phi_d \mid P_n \Leftrightarrow n\equiv a \pmod{d}$\\ \hline\hline
    $z-2$ & $1$ & 3 & (0,2), (5,6) \\
    $z-2$ & $-1$ & 4 & (0,1), (1,2), (2,6) \\ \hline
    $z^2 - z - 1$ & $1$ & 2 & (1,2), (5,6), (9,12) \\
    $z^2 - z - 1$ & $-1$ & 5 & (0,1), (0,2), (2,6), (1,3), (3,12) \\ \hline
    $z^3-z-1$ & $1$ & 1 & (0,2), (7,8), (10,12), (14,18), (21,30) \\
    $z^3-z-1$ & $-1$ & 8 & (0,1), (1,2), (1,3), (2,5), (3,8), (4,12), (5,18), (6,30) \\ \hline
    $z^3-z^2-1$ & $1$ & 1 & (0,2), (3, 4), (5, 6), (8,10), (13,18) \\
    $z^3-z^2-1$ & $-1$ & 6 & (0,1), (1,2), (2,3), (1,4), (2,6), (3,10), (4,18) \\ \hline
    $z^4-z^3-1$ & $1$ & 1 & (1,2), (0,4), (5,6), (11,14), (17,24) \\
    $z^4-z^3-1$ & $-1$ & 7 & (0,1), (0,2), (2,4), (1,5), (2,6), (3,9), (4,14), (5,24) \\ \hline
\end{tabular}
\end{table}

\section{Short polynomials of Salem numbers less than the smallest Pisot number}

In the table below we summarise what is currently known about small, short Salem numbers. The provenance of this data is as follows. The smallest known Salem number $\lambda_0=\tau_8$ was found by Lehmer \cite{Lehmer1933}. 
 The ascending sequence $\tau_n\,(n=8,9,10\cdots)$, where $\tau_n$ is the real root $>1$ of $(z^3-z-1)z^{n}+z^3+z^2-1=0$, $\{\tau_n\}$ of Salem numbers tending to the smallest Pisot number $\theta_0$ is constructed by Salem's method \cite[Theorem IV]{Salem1945}, so was presumably known to him. 
 The first table of $39$ small Salem numbers was compiled by Boyd \cite{Boyd1977}, who also found four more soon afterwards \cite{Boyd1978}. Mossinghoff \cite{Mossinghoff1998} found a further four.
These 47 Salem numbers, all less than $1.3$, are collected in Mossinghoff's table \cite{MossinghoffList} (five of them are in the sequence $\{\tau_n\}$). 
Recently Sac-\'{E}p\'{e}e \cite{Sac-Epee2025} extended the search for Salem numbers, finding 11 more Salem numbers between $1.3$ and $49/37>\theta_0-0.0004$ that are not members of the sequence $\{\tau_n\}$. Thus to the best of our current knowledge, the Salem numbers less than $\theta_0$ consist 
of the sequence $\{\tau_n\}$ and $53$ `sporadic' ones $\sigma_1,\dots,\sigma_{53}$.

In our table we present the defining polynomials of the Salem numbers in `short' form, rather than by giving their minimal polynomials. 
Where more than one such short polynomial exists, we give the one of least degree. For those Salem numbers less than $1.3$, El-Serafy and McKee \cite{El-SerafyMcKee2025} found these polynomials; as part of the current project we extended this computation to  Sac-\'{E}p\'{e}e's $11$ Salem numbers.
We observed that, while the minimal polynomial of $\sigma_{48}$, of shortness $13$, in fact divides a length-$12$ reciprocal polynomial, namely $z^{80}-z^{79}-z^{77}+z^{74}+z^{63}+z^{57}-z^{23}-z^{17}-z^6+z^3+z-1$, the quotient is not cyclotomic.
For $n\ge 18$, the generic polynomial shown for $\tau_n$ is indeed the unique short polynomial for $\tau_n$.
Any other short polynomial would either be one of the sporadic polynomials in Table \ref{Ta-Salem6}, or would be in one of the twelve infinite families: using part (vi) of Proposition \ref{T-Salemupdated} we can quickly identify any such examples.
We find that $\tau_{13}$ and $\tau_{17}$ have sporadic short polynomials, that $\tau_{16}$ has a lower-degree short polynomial coming from one of the other infinite families, and there are no such anomalous $\tau_n$ for $n\ge 18$.

The  minimal polynomials of these Salem numbers, and hence also their degrees, are readily calculated by Corollary \ref{C-cycfactor}. For the infinite family $\{\tau_n\}$, the
minimal polynomials are given by \eqref{E-min} and their degrees by $f_n$ in Proposition \ref{P-min}.

\setlength{\LTcapwidth}{\textwidth}
\newpage
\begin{longtable}[H]{|l|c|c|c|r|}
\caption{Known Salem numbers less than the smallest Pisot number $\theta_0=1.3247\cdots$}\label{Ta-AllSalem}\\  
\rot{0}{1em}{Label} &
\rot{0}{2em}{\,\,Salem number} &
\rot{30}{1em}{Degree} &
\rot{30}{1em}{Shortness} &
\rot{0}{2em}{\qquad\qquad\qquad\qquad Short polynomial of minimal degree}     \\ \hline
\endfirsthead
\caption{(continued)}\\ 
\rot{0}{1em}{Label} &
\rot{0}{2em}{\,\,Salem number} &
\rot{30}{1em}{Degree} &
\rot{30}{1em}{Shortness} &
\rot{0}{2em}{\qquad\qquad\qquad\qquad Short polynomial of minimal degree}     \\ \hline
\endhead
\endfoot
\endlastfoot
$\quad\tau_8$   & \,\,\phantom{0}\vphantom{$3^{3^3}$}1.176280818 \phantom{0} & 10 & 5 & $z^{12}-z^7-z^6-z^5+1$ \\ \hline
$\sigma_1$ & \vphantom{$3^{3^3}$}1.188368148 & 18 &  6 & $ z^{19} - z^{15} - z^{13} - z^{6} - z^{4} + 1 $  \\ \hline
$\sigma_2$ & \vphantom{$3^{3^3}$}1.200026524 & 14 & 5 & $z^{20}-z^{19}-z^{10}-z+1$   \\ \hline
$\sigma_{3}$ & \vphantom{$3^{3^3}$}1.202616744 & 14 & 5 &  $z^{14}-z^{12}-z^7-z^2+1$ \\ \hline
$\sigma_{4}$ & \vphantom{$3^{3^3}$}1.216391661 & 10 & 5& $z^{10}-z^6-z^5-z^4+1$ \\ \hline
$\sigma_{5}$ & \vphantom{$3^{3^3}$}1.219720859 & 18 &   6 & $ z^{19} - z^{17} - z^{11} - z^{8} - z^{2} + 1 $ \\ \hline
\quad$\tau_9$ & 1.230391434 &  10 & 5 & $z^{10}-z^7-z^5-z^3+1$ \\ \hline
$\sigma_{ 6}$ & \vphantom{$3^{3^3}$}1.232613549 & 20 &   6 & $ z^{25} - z^{24} - z^{16} - z^{9} - z + 1 $  \\ \hline
$\sigma_{ 7}$ & \vphantom{$3^{3^3}$}1.235664580 & 22 &   6 & $ z^{35} - z^{33} - z^{30} + z^{5} + z^{2} - 1 $  \\ \hline
$\sigma_{8}$ & \vphantom{$3^{3^3}$}1.236317932 & 16 & 5 & $z^{16}-z^{15}-z^8-z+1$ \\ \hline
$\sigma_{ 9}$ & \vphantom{$3^{3^3}$}1.237504821 & 26 &   6 & $ z^{34} - z^{32} - z^{29} - z^{5} - z^{2} + 1 $  \\ \hline
$\sigma_{ 10}$ & \vphantom{$3^{3^3}$}1.240726424 & 12 &   6 & $ z^{13} - z^{9} - z^{7} - z^{6} - z^{4} +1 $   \\ \hline
$\sigma_{ 11}$ & \vphantom{$3^{3^3}$}1.252775937 & 18 &   6 & $ z^{19} - z^{18} - z^{13} + z^{6} + z - 1 $  \\ \hline
$\sigma_{ 12}$ & \vphantom{$3^{3^3}$}1.253330650 & 20 &   6 & $ z^{20} - z^{18} - z^{15} - z^{5} - z^{2} + 1 $  \\ \hline
$\sigma_{ 13}$ & \vphantom{$3^{3^3}$}1.255093517 & 14 &   6 & $ z^{16} - z^{13} - z^{12} - z^{4} - z^{3} + 1 $  \\ \hline
$\sigma_{14 }$ & \vphantom{$3^{3^3}$}1.256221154 & 18 &    6 & $ z^{22} - z^{21} - z^{14} - z^{8} - z + 1 $  \\ \hline
$\sigma_{ 15}$ & \vphantom{$3^{3^3}$}1.260103540 & 24 &   6 & $ z^{28} - z^{27} - z^{21} - z^{7} - z + 1 $  \\ \hline
$\sigma_{16 }$ & \vphantom{$3^{3^3}$}1.260284237   & 22 & 7 & $z^{28}-z^{27}-z^{20}-z^{14}-z^8-z+1$ \\ \hline
\quad$\tau_{10}$ & \vphantom{$3^{3^3}$}1.261230961 & 10 & 5 &$z^{10}-z^8-z^5-z^2+1$ \\ \hline
$\sigma_{ 17}$ & \vphantom{$3^{3^3}$}1.263038140  & 26 & 7 & $z^{26} - z^{25} - z^{20} + z^{13} - z^6 - z + 1$ \\ \hline
$\sigma_{18 }$ & \vphantom{$3^{3^3}$}1.267296443 & 14 &   6 & $ z^{15} - z^{13} - z^{9} - z^{6} - z^{2} + 1 $  \\ \hline
$\sigma_{  19}$ & \vphantom{$3^{3^3}$}1.276779674   & 22 & 8 & $z^{28} - z^{27} - z^{20} - z^{15} - z^{13} - z^8 - z + 1$ \\ \hline
\quad$\tau_{11}$  & \vphantom{$3^{3^3}$}1.280638156 & 8 &  5 & $z^8-z^5-z^4-z^3+1$  \\ \hline
$\sigma_{ 20 }$ & \vphantom{$3^{3^3}$}1.281691372 & 26 &   6 & $ z^{27} - z^{26} - z^{21} + z^{6} + z - 1 $  \\ \hline
$\sigma_{ 21 }$ & \vphantom{$3^{3^3}$}1.282495561 & 20 &   6 & $ z^{28} - z^{27} - z^{22} + z^{6} + z - 1 $  \\ \hline
$\sigma_{ 22 }$ & \vphantom{$3^{3^3}$}1.284616551 & 18 &   6 & $ z^{34} - z^{33} - z^{28} + z^{6} + z - 1 $  \\ \hline
$\sigma_{ 23 }$ & \vphantom{$3^{3^3}$}1.284746822 & 26 &   6 & $ z^{35} - z^{34} - z^{29} + z^{6} + z - 1 $  \\ \hline
$\sigma_{ 24 }$ & \vphantom{$3^{3^3}$}1.285099364 & 30 &   6 & $ z^{41} - z^{40} - z^{35} + z^{6} + z - 1 $  \\ \hline
$\sigma_{25 }$ & \vphantom{$3^{3^3}$}1.285121520 & 30 &   6 & $ z^{42} - z^{41} - z^{36} + z^{6} + z - 1 $  \\ \hline
$\sigma_{ 26 }$ & \vphantom{$3^{3^3}$}1.285185671 & 30 &   6 & $ z^{49} - z^{48} - z^{43} + z^{6} + z - 1 $  \\ \hline
$\sigma_{27 }$ & \vphantom{$3^{3^3}$}1.285196727 & 26 &   6 & $ z^{56} - z^{55} - z^{50} + z^{6} + z - 1 $  \\ \hline
$\sigma_{ 28 }$ & \vphantom{$3^{3^3}$}1.285199179 & 44 &   6 & $ z^{67} - z^{66} - z^{61} - z^{6} - z + 1 $  \\ \hline
$\sigma_{ 29}$ & \vphantom{$3^{3^3}$}1.285235436 & 30 &   6 & $ z^{45} - z^{44} - z^{39} - z^{6} - z + 1 $  \\ \hline
$\sigma_{ 30 }$ & \,\phantom{0}\vphantom{$3^{3^3}$}1.285409065\phantom{0} & 34 &   6 & $ z^{38} - z^{37} - z^{32} - z^{6} - z + 1 $  \\ \hline
$\sigma_{ 31 }$ & \vphantom{$3^{3^3}$}1.286395967 & 18 &   6 & $ z^{19} - z^{18} - z^{12} - z^{7} - z + 1 $  \\ \hline
$\sigma_{ 32 }$ & \vphantom{$3^{3^3}$}1.286730182 & 26 &   6 & $ z^{30} - z^{29} - z^{24} - z^{6} - z + 1 $  \\ \hline
$\sigma_{ 33 }$ & \vphantom{$3^{3^3}$}1.291741426 & 24 &   6 & $ z^{24} - z^{23} - z^{18} - z^{6} - z + 1 $\\ \hline
$\sigma_{ 34 }$ & \vphantom{$3^{3^3}$}1.292039106     & 20 & 7 & $z^{24} - z^{23} - z^{17} - z^{12} - z^7 - z  + 1$ \\ \hline
$\sigma_{35 }$ & \vphantom{$3^{3^3}$}1.292418658   & 40 & 10 & $z^{42} - z^{40} - z^{39} + z^{29} + z^{28} - z^{14} - z^{13} + z^3 + z^2 - 1$  \\ \hline 
$\sigma_{ 36 }$ & \vphantom{$3^{3^3}$}1.292900722   & 46 & 10 & $z^{47} - z^{46} - z^{43} + z^{39} + z^{26} - z^{21} - z^8 + z^4 + z - 1$ \\ \hline
\quad$\tau_{12}$ & \vphantom{$3^{3^3}$}1.293485953 & 10 &  5 & $z^{12}-z^{11}-z^6-z+1$ \\ \hline
$\sigma_{ 37 }$ & \vphantom{$3^{3^3}$}1.295675372 & 18 &   6 & $ z^{22} - z^{21} - z^{16} - z^{6} - z + 1 $ \\  \hline
$\sigma_{ 38 }$ & \vphantom{$3^{3^3}$}1.296210660   & 34 & 12 & $z^{38} - z^{36} - z^{34} - z^{29} + z^{28} - z^{24} - z^{14} + z^{10} - z^9 - z^4 - z^2 + 1$ \\ \hline
$\sigma_{ 39 }$ & \vphantom{$3^{3^3}$}1.296421365  & 22 & 7 & $z^{22}-z^{21} - z^{17} + z^{11} - z^5 - z + 1$ \\ \hline
$\sigma_{ 40 }$ & \vphantom{$3^{3^3}$}1.296821374   & 28 & 8 & $z^{37} - z^{36} - z^{33} + z^{29} - z^8 + z^4 + z - 1$ \\ \hline
$\sigma_{ 41 }$ & \vphantom{$3^{3^3}$}1.298429835   & 36 & 10 & $z^{39} - z^{38} - z^{35} + z^{31} - z^{20} + z^{19} - z^8 + z^4 + z - 1$ \\ \hline
$\sigma_{ 42 }$ & \vphantom{$3^{3^3}$}1.299744869  & 26 & 10 & $z^{34} - z^{33} - z^{27} - z^{22} - z^{18} - z^{16} - z^{12} - z^7 - z +1$ \\  \hline
\quad$ \tau_{13} $  &  \vphantom{$3^{3^3}$}1.302268805  & 12 & 6 &   $z^{13}-z^{11}-z^{8} - z^5-z^2+1$ \\ \hline
$\sigma_{43} $ & $ 1.302721444$   & 32 & 10 & $\vphantom{2^{2^{2^2}}}z^{40} -z^{39} -z^{36} +z^{32} -z^{22} +z^{18} -z^8 +z^4 +z - 1$ \\ \hline
$\sigma_{44} $ & $1.303283349$ & 32 & 8 & $\vphantom{2^{2^{2^2}}}z^{37} -z^{35} -z^{34} +z^{25} -z^{12} +z^3 +z^2 - 1$ \\ \hline
$\sigma_{45} $ & $1.303385419$ & 30 & 10 & $\vphantom{2^{2^{2^2}}}z^{34} -z^{33} -z^{29} +z^{23} -z^{18} -z^{16} +z^{11} -z^5 -z + 1$ \\  \hline
$\sigma_{46} $ & $1.304697625$ & 26 & 9 & $\vphantom{2^{2^{2^2}}}z^{26} -z^{24} -z^{23} +z^{15} -z^{13} +z^{11} -z^3 -z^2 + 1$ \\ \hline
$\sigma_{47} $ & $1.305131379$ & 22 & 9 & $\vphantom{2^{2^{2^2}}}z^{22} -z^{21} -z^{18} +z^{15} -z^{11} +z^7 -z^4 -z + 1$ \\ \hline
$\sigma_{48} $ & \vphantom{$3^{3^3}$}$1.306473538$ & 38 & 13 & $\vphantom{2^{2^{2^2}}}z^{44}\! -\!z^{43}\! -\!z^{37}\! -\!z^{33}\! -\!z^{30} -z^{24} +z^{22} -z^{20} -z^{14} -z^{11} -z^7 -z + 1$ \\ \hline
$\sigma_{49} $ & $1.308071086 $ & 44 & 10 & $\vphantom{2^{2^{2^2}}} z^{50} -z^{48} -z^{47} +z^{35} +z^{34} -z^{16} -z^{15} +z^3 +z^2 - 1$ \\ \hline
\quad$ \tau_{14} $ & \vphantom{$3^{3^3}$}1.308409006 & 16 &   6 &  $z^{17}-z^{15}-z^{14} + z^3+z^2-1$ \\ \hline
$\sigma_{50} $ & $1.308966300 $ & 36 & 10 & $\vphantom{2^{2^{2^2}}}z^{42} -z^{40} -z^{39} +z^{27} +z^{26} -z^{16} -z^{15} +z^3 +z^2 - 1$ \\ \hline
$\sigma_{51} $ & $1.310180863 $ & 22 & 7 & $\vphantom{2^{2^{2^2}}}z^{26} -z^{24} -z^{23} +z^{13} -z^3 -z^2 + 1$ \\ \hline
$\sigma_{52} $ & $1.312566633 $ & 22 & 8 & $\vphantom{2^{2^{2^2}}}z^{42} -z^{40} -z^{39} +z^{28} +z^{14} -z^3 -z^2 + 1$ \\ \hline
\quad$ \tau_{15} $ & \vphantom{$3^{3^3}$}1.312773240  & 16 &  6 &   $z^{18}-z^{16}-z^{15} + z^3+z^2-1$ \\ \hline
\quad$ \tau_{16} $ &  \vphantom{$3^{3^3}$}1.315914432 & 12 &  6 &   $z^{13}-z^{12}-z^{9} + z^4+z-1$ \\ \hline
$\sigma_{53} $ & $1.316069253 $ & 40 & 10 & $\vphantom{2^{2^{2^2}}}z^{48} -z^{46} -z^{45} +z^{31} +z^{30} -z^{18} -z^{17} +z^3 +z^2 - 1$ \\ \hline
\quad$ \tau_{17} $ & \vphantom{$3^{3^3}$}1.318197504 & 14 & 6 & $z^{17}-z^{16}-z^{11}-z^6-z+1$ \\ \hline
\multicolumn{2}{|l|}{ \quad$\tau_{n}$ $\text{ for }n\ge 18$ } & $f_n$   &  6 & \vphantom{$3^{3^3}$}$z^n(z^3-z-1) +z^3+z^2-1$
\\ \hline
\end{longtable}

\end{document}